\documentclass[12pt]{article}
\usepackage{graphicx} 

\usepackage{tikz}

\usepackage{amsmath,amsthm,amscd,amssymb,eucal,mathrsfs}
\usepackage{amsfonts}
\usepackage{latexsym}
\usepackage{graphicx} 
\usepackage{mathtools}
\usepackage{multicol}
\usepackage{dsfont}
\usepackage[all]{xy}
\usepackage{multirow}
\usepackage{color}
\usepackage{mathtools}
\usepackage[hidelinks]{hyperref}
\usepackage[all]{xy}

\newtheorem{thm}{Theorem}[section]

\newtheorem{cor}[thm]{Corollary}
\newtheorem{remark}[thm]{Remark}
\newtheorem{definition}[thm]{Definition}

\newtheorem*{Satz*}{Satz}

\newtheorem{Lemma}[thm]{Lemma}

\newtheorem{proposition}[thm]{Proposition}

\newtheorem{Corollary}[thm]{Corollary}

\newcommand{\mathset}[1]{{\left\{#1\right\}}}
\newcommand{\absolute}[1]{\left\lvert#1\right\rvert}
\newcommand{\norm}[1]{\left\|#1\right\|}

\DeclareMathOperator{\PGL}{PGL}

\DeclareMathOperator{\diam}{diam}

\DeclareMathOperator{\GL}{GL}

\DeclareMathOperator{\supp}{supp}
\DeclareMathOperator{\Gal}{Gal}

\DeclareMathOperator{\rig}{rig}
\DeclareMathOperator{\unr}{nr}
\DeclareMathOperator{\tr}{tr}
\DeclareMathOperator{\Ran}{Ran}

\title{Schottky invariant diffusion on the transcendent $p$-adic upper half plane}
\author{Patrick Erik Bradley}
\date{\today}

\begin{document}

\maketitle

\begin{abstract}
The transcendent part of the Drinfel'd $p$-adic upper half plane is shown to be a Polish space. Using Radon measures associated with regular differential $1$-forms invariant under Schottky groups allows to construct self-adjoint  diffusion operators as Laplacian integral operators with kernel functions determined by the $p$-adic absolute value on the complex $p$-adic numbers. Their spectra are explicitly calculated and the corresponding Cauchy problems for their associated heat equations are found to be uniquely solvable and to determine Markov processes having paths which are c\`adl\`g. The heat kernels are shown to have explicitly given distribution functions, as well as    boundary value problems associated with the heat equations under Dirchlet and von Neumann conditions are solved.
\end{abstract}

\section{Introduction}

Diffusion on $p$-adic domains is of increasing interest, both from a theoretical as well as from a practical perspective. On the practical side, the local tree structure of such spaces seems appealing, and ultrametricity is found e.g.\ in spin glasses as well as in energy landscapes in various application domains
in the natural sciences and applied mathematics, in particular where complex networks are relevant \cite{HH2021,ZunigaNetworks,Zuniga2022} and in order to model stochastic processes like $p$-adic Brownian motion \cite{RW2023,Zuniga2015,Weisbart2021,Weisbart2024,PRSWY2024}.
In this context, a strong motor for developing ultrametric analysis is given by theoretical physics and its relationship with number theory
\cite{Varadarajan1997,Varadarajan2002,HMSS2018,Marcolli2020,MM2020}.
\newline

On the other hand, number-theoretic questions are an invitation to study problems over local fields
and other non-archimedean domains, where analytic tools for stochastic processes lead to operators acting on complex-valued functions. The work by S.\ Haran on their relationship with the Riemann zeta function is an example of such an arithmetic consideration, cf.\ 
\cite{Haran1990}.
The arithmetic significance of $p$-adic analysis is beginning to play out itself further through the study of Markov processes invariant under hyperbolic discontinuous groups, 
as effected e.g.\ in \cite{Yasuda2004} on a space named 
\emph{$p$-adic half-plane}.
This space is different from
Drindfeld's $p$-adic upper half plane
\cite{DT2007} which allows a $p$-adic uniformisation of Shimura curves, cf.\ also \cite{BC1991,Kuennemann2000}.
This reveals them as disjoint unions of Mumford curves, and diffusion on the points of such curves
defined over a $p$-adic number field
has been initiated the author in \cite{brad_heatMumf,brad_thetaDiffusionTateCurve,Brad_HearingGenusMumf,
BL_shapes_p}.
\newline

One unsettled question is how to carry over the ultrametric analytic methods developped over local fields to the field $\mathds{C}_p$ of $p$-adic complex numbers, which is not locally compact and thus has no Haar measure available for such a task, cf.\ \cite[Ch.\ 6.8]{Gouvea2020}. One aim of this present article is to show that at least on the transcendent part $\Omega_{\tr}$ of $\mathds{C}_p$, called here the \emph{transcendent $p$-adic upper half plane}, this is possible due to the work \cite{APZ1998}, where it is shown that this transcendent part is locally profinite, where these local pieces are given by orbits of transcendent $p$-adic numbers under the continuous action of the $p$-adic absolute Galois group.
Here, it will be shown that 
$\Omega_{\tr}$ is a Polish space and allows for Radon measures with which ultrametric Diffusion operators acting on function spaces on the transcendent $p$-adic upper half plane can be constructed, modelled after the $p$-adic Vladimirov and Taibleson operators \cite{Taibleson1975,VVZ1994,RZ2008} and the Z\`u\~niga operators from \cite{ZunigaNetworks} in a fashion which is invariant under the Schottky groups forming the $p$-adic uniformisation of a Shimura curve. This approach generalises \cite{brad_SchottkyDiffusion}, where the operators were defined on the $p$-adic points of a Mumford curve. 
\newline

Unlike in the case of Z\`u\~niga operators, which have eigenvalues with infinite multiplicity, cf.\ \cite[Ex.\ 1]{brad_heatMumf} and \cite[Prop.\ 4]{BL-topoIndex_p},
the present operators generate a  semigroup having a heat kernel distribution function which furthermore can be made explicit in terms of the eigenfunctions. In $p$-adic  diffusion on finite topologies, the heat kernel was exhibited as an explicitly defined distribution, cf.\ \cite[Cor.\ 3.7]{BL-topoIndex_p}. But the method of this article also applies to that case, which leads to another class of $p$-adic Laplacian operators generating a diffusion having a heat kernel distribution function.
In the end, boundary value problems for the heat equation under invariant Dirichlet or invariant von Neumann conditions are solved explicitly.
\newline

The main results are given concretely as follows:
\newline

\noindent
{\bf Theorem \ref{countable}.}
\emph{The transcendent $p$-adic upper half plane $\Omega_{\tr}$ is locally profinite, and the set of orbits of transcendental elements in $\mathds{C}_p$ under the action of the absolute $p$-adic Galois group is countably infinite.}
\newline

A consequence is that the transcendent $p$-adic upper half plane is a Polish space (Corollary \ref{Polish}).
\newline

Theorem \ref{MumfordEV} states that the $L^2$-space of functions on $\Omega_{\tr}$ outside a measure zero part determined by a given regular differential form $\omega$ invariant under the Schottky group has an orthonormal basis consisting of invariant ultrametric wavelets which are eigenfunctions of the diffusion operator $\Delta_\alpha^{\frac12}$ constructed in this article. The corresponding eigenvalues can be calculated explicitly, have finite multiplicity and an invariance property under the Schottky group.
\newline 

Theorem \ref{ShimuraEigenbasis} states that the $L^2$-space associated with a $p$-adically uniformised Shimura curve has an orthonormal eigenbasis for a certain self-adjoint and negative definte Laplacian integral operator $\mathcal{L}$ constructed from a collection of data coming from the constitutent Mumford curves. Again, the eigenvalue multiplicities are finite.
\newline

Theorem \ref{heatKernelDistribution}
states that the semigroup associated with the operator $\mathcal{L}$ yields a unique solution to the Cauchy problem for the heat equation associated with $\mathcal{L}$ and induces an invariant Markov process 
on the underlying ultrametric space outside the measure zero part determined by the given Schottky invariant regular differential forms.
Its paths are c\`adl\`ag.
\newline

\noindent
{\bf Corollary \ref{heatKernel}.} \emph{The heat equation associated with $\mathcal{L}$ has a well-defined heat kernel function $p(t,x,y)$ for $t\ge0$ and $x\neq y$. For $x=y$ it is only defined for $t>>0$.}
\newline

\noindent
{\bf Theorem \ref{bvp_Shimura}.}
\emph{The boundary value problem
for the heat equation associated with $\mathcal{L}$ having as initial condition an invariant test function and either a Dirichlet or a von Neumann boundary condition on a compact open orbit $\Gamma S$ has a solution for $t\ge0$ if and only if the initial function is supported in $\Gamma S$.}
\newline

The article is structured into three further sections. Section 2 gives a brief review of $p$-adic uniformisation and states the needed results about the Drinfel'd upper half plane, Mumford curves and the uniformisation of Shimura curves. Section 3 studies the transcendent $p$-adic upper half plane, provides with appropriate Radon measures and constructs new integral operators in order to study their spectral behaviour. Section 4 is devoted to the study of the heat equation on the transcendent part of $p$-adically uniformised Shimura curves using these new invariant operators and includes boundary value problems.

\section{A brief review of $p$-adic uniformisation}

Here, the well-known $p$-adic uniformisation of Shimura curves is briefly explained by introducing the Drinfel'd $p$-adic upper half plane, $p$-adic uniformisation of
Mumford curves and how this concept extends to Shimura curves.

\subsection{The Drinfel'd $p$-adic upper half plane}

The $p$-adic counterpart $\mathds{C}_p$ of the complex numbers is the smallest algebraically closed and complete extension field of the $p$-adic number field $\mathds{Q}_p$. It turns out that this can be obtained by completion of the algebraic closure $\bar{\mathds{Q}}_p$ of $\mathds{Q}_p$, and that this is of infinite dimension, when viewed as a vector space over $\mathds{Q}_p$, and is not locally compact, cf.\  \cite[Ch.\ 6.8]{Gouvea2020}.
\newline

The $p$-adic upper half plane over the field of $p$-adic numbers  is defined as
\[
\Omega_p=\mathds{P}^1(\mathds{C}_p)\setminus\mathds{P}^1(\mathds{Q}_p)\,
\]
where $\mathds{P}^1$ denotes the projective line. The space $\Omega_p$
 consists of the complex $p$-adic points of the rigid-analytic upper half plane studied by Drinfel'd, cf.\ \cite{Drinfeld1974,DT2007}.
Important for the remainder of the article is that $\PGL_2(\mathds{Q}_p)$ acts on $\Omega_p$ by M\"obius transformations. Also, the $p$-adic absolute Galois group $G=\Gal(\bar{\mathds{Q}}_p/\mathds{Q}_p)$ acts on $\Omega_p$ as continuous isometries.

\subsection{Mumford curves}

A short introduction to Mumford curves can be found in 
\cite{Schmechta2000}. More detailed treatments are \cite{GvP1980} and
\cite{FP2004}. Here, a few facts are collected, some of which will be used in the remainder of this article. 
\newline

A $p$-adic Mumford curve is a projective algebraic curve $X$ defined over $\mathds{Q}_p$ such that it is the quotient
\[
X=\Gamma\setminus\Omega
\]
under the action of a finitely generated free discrete subgroup $\Gamma$ of $\PGL_2(\mathds{Q}_p)$.
The group $\Gamma$ is called \emph{Schottky group}, and its non-trivial elements are  hyperbolic M\"obius transformations. The action takes place on an open subspace $\Omega$ of the $p$-adic projective line $\mathds{P}^1$ outside the set of limit points of $\Gamma$, and is discontinuous. The number of generators of $\Gamma$ coincides with the genus of $X$.
\newline

The important theorem is:
\begin{thm}[Mumford]
The conjugacy classes of Schottky groups within $\PGL_2(\mathds{Q}_p)$ are in one-to-one correspondence with the isomorphism classes of $p$-adic Mumford curves.
\end{thm}

\begin{proof}
Cf.\ \cite{Mumford1972}. A $p$-adic analytic proof is contained in \cite{GvP1980}.
\end{proof}

After a finite field extension, a Mumford curve can be related to a so-called \emph{stable} graph, which is a connected finite graph having no vertex of degree at most $2$, and  whose first Betti number equals to the genus of $X$. It is the quotient under the action of $\Gamma$ on a certain subtree of the $p$-adic Bruhat-Tits tree $\mathscr{T}_p$, and is called the \emph{reduction graph} of $X$. Cf.\ \cite{Serre1980} for an introduction to the $p$-adic Bruhat-Tits tree.

\subsection{$p$-adic uniformisation of Shimura curves}

Let $\Delta$ be a quaternion skew field having centre $\mathds{Q}$, and let $\delta\in\mathds{N}$ be the product of all prime numbers where $\Delta$ is ramified. Let $S_N$ be the generic fibre of the Shimura curve which represents the moduli functor for certain abelian surfaces over $\mathds{Z}_p$-schemes. The details of this functor can be found e.g.\ in \cite[Thm.\ 1]{Kuennemann2000}, but are not important for this article. What is important, however, is that it is a projective flat $\mathds{Z}_p$-scheme, and its generic fibre is the Shimura curve defined over $\mathds{Q}_p$ we are interested in.
For more details on Shimura curves and the following \v{C}erednik-Drinfel'd theorem, cf.\ e.g.\ \cite{Kuennemann2000} (a short exposition), or \cite{BC1991} (a longer exposition).

\begin{thm}[\v{C}erednik-Drinfel'd]
There is a canonical isomorphism of rigid analytic spaces over $\mathds{Q}_p$:
\[
S_N^{\rig}=\GL_2(\mathds{Q}_p)\setminus \left[\Omega_K\otimes_{\mathds{Q}_p}\mathds{Q}_p^{\unr}\times Z_N\right]
\]
for $p\mid\delta$.
\end{thm}

Here, the $p$-adic upper half plane
\[
\Omega_K=\mathds{P}^1(\mathds{C}_p)\setminus\mathds{P}^1(K)
\]
over some $p$-adic number field $K$ unramified over $\mathds{Q}_p$ is used, and $Z_N$ is a certain space derived from $\Delta$ contributing  to only finitely many orbits of a pair $(w,z)$ with fixed $w\in\Omega_K\otimes_{\mathds{Q}_p}\mathds{Q}_p^{\unr}$  under the action of $\GL_2(\mathds{Q}_p)$, and it thus happens to be that
\[
S_N^{\rig}=\bigsqcup\limits_{i=1}^N\Gamma_i'\setminus\Omega_K
\]
for some Schottky subgroups  $\Gamma_i'\subset\PGL_2(\mathds{Q}_p)$ with $i=1,\dots,N$---this is immediate from  \cite[sentence containing (3.7)]{Kuennemann2000}.
Hence, $S_N^{\rig}$ is a finite disjoint union of Mumford curves.
\newline

The transcendent points of the Shimura curve $S_N^{\rig}$ are consequently described as
\begin{align}\label{GivenShimura}
S^{\rig}_{N,\tr}=S_N^{\rig}(\mathds{C}_p)\setminus S_N^{\rig}(\bar{\mathds{Q}}_p)
=\bigsqcup\limits_{i=1}^N\Gamma_i'\setminus\left(\Omega(\mathds{C}_p)\setminus\Omega(\bar{\mathds{Q}}_p)\right)
=\bigsqcup\limits_{i=1}^N \Gamma_i'\setminus\Omega_{\tr},
\end{align}
where
\[
\Omega_{\tr}=\Omega(\mathds{C}_p)\setminus\Omega(\bar{\mathds{Q}}_p)
\]
is the transcendent $p$-adic upper half plane.
\newline

In the following, the constituent Mumford curves are assumed to be  each given as $\Gamma_i'\setminus \Omega_p$ for $i=1,\dots,N$.

\section{The transcendent $p$-adic upper half plane}

Since the field of $p$-adic complex numbers $\mathds{C}_p$ is not locally compact, there is no Haar measure to work with. However, it has been shown by Alexandru, Zaharescu and Popescu that that the transcendent part $\Omega_{\tr}$ of $\mathds{C}_p$ is locally profinite, cf.\ \cite{APZ1998}. It turns out that $\Omega_{\tr}$ is an irregular ultrametric analogue of  $\mathds{Q}_p$ in the sense that it is Polish, since it is covered by countably many profinite sets, and thus allows for Radon measures in such a way that Schottky invariant linear diffusion operators can be constructed on function spaces over $\Omega_{\tr}$.

\subsection{$\Omega_{\tr}$ is locally profinite}

From \cite[1.\S4 and 2.\S1]{APZ1998}, it can be easily seen that the following is an equivalence relation on the set $\mathscr{D}$ of all distinguished sequences in $\mathds{C}_p$, namely:
\[
(\alpha_n)\sim(\beta_n):\Leftrightarrow
(\alpha_n-\beta_n)\in\mathscr{D}\;\text{and}\;v(\alpha_n-\beta_n)\to\infty,
\]
where $v$ is the valuation of $\mathds{C}_p$. Define
for $\rho\in\mathds{Q}$ the set
$\mathscr{D}_\rho\subset\mathscr{D}/\!\sim$ as the set of equivalence classes of distinguished sequences $(\alpha_n)$ with $v(\alpha_n)=\rho$ for $n>>0$.

\begin{thm}\label{countable}
The transcendent $p$-adic upper half plane $\Omega_{\tr}$ is locally pro-finite, and
the set of orbits of transcendental elements in $\mathds{C}_p$ under the action of the absolute $p$-adic Galois group  is countably infinite.
\end{thm}

\begin{proof}
The locally pro-finite property is a consequence of \cite[Thm.\ 3.5]{APZ1998}.
An equivalence class $[(\alpha_n)]\in\mathscr{D}_0$ can be represented by a sequence of elements in the residue field of a $p$-adic number field.
Since $\mathds{Q}_p$  has only finitely many extension fields of degree $\le n$, it follows that a given  a prefix 
\[
[\alpha_0],\dots,[\alpha_n]
\]
of some $[(\alpha_n)]\in\mathscr{D}_0$ has only finitely many possibilities of continuation in the next place $n+1$,
because any possible class of $[\alpha_{n+1}]$ corresponds to  a unique element of some residue field at dimension $\dim(\alpha_{n+1})$, and this is a finite field.
This shows that the set $S_0^{\tr}$ of transcendent elements of the unit circle in $\mathds{C}_p$ is the boundary of a locally finite rooted tree, and also that the Galois orbit $\mathcal{C}(t)$ for $t\in S_0^{\tr}$ is the boundary of a branch starting in some class $[\alpha_0]$. 
The topology induced by this tree coincides with the induced topology from $\mathds{C}_p$. In fact this is the content of the proof of \cite[Thm.\ 3.5]{APZ1998}. 
Hence, $S_0^{\tr}$ is a compact locally pro-finite space, i.e.\ a finite union of Galois orbits. Via rescaling, this property also  belongs to the transcendent points  $S_\rho^{\tr}$ of the complex $p$-adic circle of radius $p^\rho\in p^{\mathds{Q}}$, only that the number of disjoint Galois orbits covering $S_\rho^{\tr}$ might be larger, because the branches start in a higher-dimensional extension field of $\mathds{Q}_p$. But it  still is finite.
The countability of the set $p^{\mathds{Q}}$ of possible radii of circles now proves the assertion.
\end{proof}

\begin{Corollary}\label{Polish}
The space $\Omega_{\tr}$ is a locally compact Polish subspace of $\mathds{C}_p$.
\end{Corollary}

\begin{proof}
The subspace $\Omega_{\tr}$ of $\mathds{C}_p$ is locally compact, since it is locally pro-finite according to
\cite[Thm.\ 3.5]{APZ1998}.
Furthermore, it is clearly metrisable and, by Theorem \ref{countable}, countable at infinity. This means that it is also Polish.
\end{proof}

\begin{remark}
The transcendent $p$-adic upper half plane $\Omega_{\tr}$ 
can be viewed as an ultrametric manifold having a countable covering by ultrametric discs given by the Galois orbits $\mathcal{C}(t)$.
\end{remark}

\subsection{Measures on $\Omega_{\tr}$}
\label{measures}

The Haran correspondence is a bijection between  Markov chains on a rooted tree $T$ and probability measures on $\partial T$ \cite[Ch.\ 2.1.1]{Haran2008}.
By multiplying with a fixed number $r>0$, a Markov chains corresponds now to a Radon measure $\nu$ on $\partial T$ with $\nu(\partial T)=r$.
In this manner, take now
on each $\mathcal{C}(t)$ a Radon measure in a way such that the measure of each vertex equals the sum of the child vertex measures.
A constructive  proof of the existence of  such a so-called \emph{equity measure}  can be found in \cite[\S 2.2]{LocalUltrametricity}, and this is an example of such a Radon measure.
\newline

Define now a measure $\nu$ on $\Omega_{\tr}$ by a collection of equity measures on each $\mathcal{C}(t)$.
A possible realisation is to say that 
\begin{align}\label{localProbability}
\nu(\mathcal{C}(t))=1
\end{align}
for all $t\in\Omega_{\tr}$, i.e.\ to realise such a measure as probability measure. 

\begin{definition}
A measure $\nu$ as above is called a \emph{local equity measure} on the transcendent $p$-adic upper half plane $\Omega_{\tr}$. In case (\ref{localProbability}) holds true, $\nu$ is called a \emph{local equity probability measure}.
\end{definition}

Let $\Gamma\subset\PGL_2(\mathds{C}_p)$ be a Schottky group, and $F\subset\Omega_K$ a good fundemantal domain for the action of $\Gamma$ on its domain of regularity being $\Omega_K$. Then define
\[
F_{\tr}=F\cap\Omega_{\tr}
\]
which is a compact subset of $\Omega_{\tr}$.

\begin{Lemma}
Assume that $F$ is defined over $\mathds{Q}_p$, i.e.\ the centres of the discs defining $F$ can be chosen in $\mathds{Q}_p$. Then $F$ is invariant under the $p$-adic absolute Galois group.
\end{Lemma}

\begin{proof}
Given an inequality
\[
\absolute{x-a}\le\absolute{\epsilon}
\]
with $a\in\mathds{Q}_p$, $\epsilon\in\mathds{C}_p$, observe that
\[
\absolute{x^\sigma-a}=\absolute{x^\sigma-a^\sigma}=\absolute{\sigma(x-a)}=\absolute{x-a}
\]
for any $\sigma\in\Gal(\overline{\mathds{Q}}_p/\mathds{Q}_p)$. This proves the Galois invariance of discs centred in points of $\mathds{Q}_p$. This now proves the assertion.
\end{proof}

\begin{proposition}\label{GammaLocalEquity}
Let $\Gamma$ be a Schottky group acting on $\Omega_{K}$. Assume that $F\subset\Omega_K$ is a good fundamental domain defined over $\mathds{Q}_p$. Then 
there exists a $\Gamma$-invariant local equity  measure $\mu$ on $\Omega_{\tr}$ with 
$\mu(F_{\tr})<\infty$.
\end{proposition}

\begin{proof}
The compact $F_{\tr}$ is the union of finitely many Galois orbits $\mathcal{C}(t)$. On each of these orbits define an equity probability measure, and together these define a local equity probability measure $\nu_F$ on $F_{\tr}$. Now, define for $\gamma\in\Gamma$ the measure
\[
\nu_{\gamma F}=\absolute{\gamma'}\nu_F\,,
\]
and obtain a local equity measure with the transformation rule
\begin{align}\label{traforule}
\int_{\Omega_{\tr}} g(\gamma(x))\nu_F(\gamma(x))=\int_{\Omega_{\tr}} g(x)\absolute{\gamma'(x)}\nu_F(x)
\end{align}
for $\Gamma$-invariant integrable functions $g\colon\Omega_{\tr}\to\mathds{C}$. Let $\omega\in\Omega^1(\Omega_p)$ be a regular $\Gamma$-invariant analytic differential $1$-form on $\Omega_p$. Since $\Gamma\setminus\Omega_p$ is a Mumford curve, such exists, and on $F$ it has the form
\[
\omega|_F=f_F\,dx\,,
\]
where $f_F$ is a holomorphic function $F\to\mathds{C}_p$. Then define
\[
\absolute{\omega}_{\tr,F}:=\absolute{f_F}\nu_F
\]
which due to the transformation rule (\ref{traforule}) and the $\Gamma$-invariance of $\omega$,
expressing itself as
\[
f_{\gamma F}(x)=f_F(\gamma(x))=\frac{f_F(x)}{\gamma'(x)}\,,
\]
extends itself to a $\Gamma$-invariant local equity measure $\absolute{\omega}_{\tr}$. Since
\[
\int_{F_{\tr}}\absolute{\omega}_{\tr}<\infty,
\]
it follows that taking $\mu=\absolute{\omega}_{\tr}$
proves the assertion.
\end{proof}

The $\Gamma$-invariant local equity measure $\absolute{\omega}_{\tr}$ defined in the proof of Proposition \ref{GammaLocalEquity}
 depends on the choice of a regular $\Gamma$-invariant analytic differential $1$-form $\omega$ on $\Omega_p$ and a local equity probability measure on the transcendental part $F_{\tr}$ of a fundamental domain $F$. Since the Mumford curve $\Gamma\setminus\Omega_p$ is a projective algebraic curve defined over $\mathds{Q}_p$, it follows that regular differential $1$-forms on it are algebraic.  Hence
the zeros of $\omega$
are $p$-adic algebraic numbers, and thus are already excluded 
from the transcendent $p$-adic upper half plane $\Omega_{\tr}$ on which the measure $\absolute{\omega}_{\tr}$ is defined.
\newline

The construction of the measure $\absolute{\omega}_{\tr}$ compares with the earlier construction on the $p$-adic points of a Mumford curve $\Gamma\setminus\Omega_p$ in \cite{brad_SchottkyDiffusion}, as both use a fixed regular $\Gamma$-invariant analytic differential $1$-form on the domain of regularity of $\Gamma$. 
When constructing the diffusion operator in \cite{brad_SchottkyDiffusion}, it was helpful to leave out the vanishing locus of $\omega$. In the following sections, this step will become unnecessary, as that locus consists of algebraic $p$-adic numbers only.

\subsection{Invariant operators on the transcendent $p$-adic half plane}

Corollary \ref{countable} showed that $\Omega_{\tr}$ is a countable union of Galois orbits, the theory of Schottky-invariant Z\'u\~niga-type operators from \cite{brad_SchottkyDiffusion}
can be adapted to the situation of the transcendent $p$-adic points of a Mumford curve, i.e.\ the points of $\Omega_{\tr}$ invariant under a Schottky group $\Gamma\subset\PGL_2(\mathds{Q}_p)$ freely generated by $g>0$ hyperbolic transformations.
\newline

Let $\mu$ be a $\Gamma$-invariant equity measure on $\Omega_{\tr}$, which exists according to Proposition \ref{GammaLocalEquity}.
Just like in \cite{brad_SchottkyDiffusion}, let $\alpha>0$ such that
\begin{align}\label{assumption_p}
p^\alpha>2g
\end{align}
where  the genus of the Mumford curve
\[
X=\Gamma\setminus\Omega
\]
coincides with $g$.
\newline

Since the group $\Gamma$ is a finitely generated free group, one can fix a set of generators and their inverses as an alphabet, and then view every element of $\Gamma$ as a word over this alphabet. The length $\ell(\gamma)$ of $\gamma\in\Gamma$ is then defined as the number of letters occuring in the word after reduccing all neighbouring pairs of the form
$gg^{-1}$ and $g^{-1}g$, where $g$ is a letter of this alphabet.
\newline

Now define
\[
H_\alpha(\beta x,\gamma y)
=\mu(F_{\tr})^{-1}p^{-\alpha\ell(\beta^{-1}\gamma)}\absolute{\beta x-\gamma y}^{-\alpha}\,,
\]
where $F\subset\Omega$ is a good fundamental domain, and $x,y\in F$, $\beta,\gamma\in\Gamma$. The corresponding Laplacian operator is
\[
\mathcal{H}_\alpha(\beta x)=\sum\limits_{\gamma\in\Gamma}\int_{F_{\tr}} H_\alpha(\beta x,\gamma y)(u(y)-u(x))\mu(y)
\]
where $u\in\mathcal{D}(\Omega_{\tr})^\Gamma$ is a $\Gamma$-invariant locally constant function on $\Omega_{\tr}$. This defines an operator
\[
\mathcal{H}_\alpha\colon \mathcal{D}(\Omega_{\tr})^\Gamma\to
\mathcal{D}(\Omega_{\tr})
\]
for $\alpha>0$. There is now a Dirichlet bilinear form
\begin{align*}
\mathcal{E}_\alpha(u,v)
&=
\langle\mathcal{H}_\alpha u,\mathcal{H}_\alpha v\rangle
\\
&=\sum\limits_{\beta,\gamma\in\Gamma}\iint_{F_{\tr}^2}
H_\alpha(\beta x,\gamma y)(u(y)-u(x))\left(\overline{v(y)}-\overline{v(y)}\right)\,\mu(y)
\end{align*}
and a quadratic Dirichlet form
\[
\mathcal{E}_\alpha(u)=
\langle\mathcal{H}_\alpha u,\mathcal{H}_\alpha\rangle
\]
for $\alpha>0$. 
\begin{Lemma}\label{denselyDefined}
The operator $\mathcal{H}_\alpha$ and the Dirichlet form $\mathcal{E}_\alpha$ are both defined on all of $\mathcal{D}(\Omega_{\tr})^\Gamma$ for $\alpha>0$ satisfying (\ref{assumption_p}).
\end{Lemma}

\begin{proof}
Let $u\in\mathcal{D}(\Omega_{\tr})^\Gamma$. Then, similarly as in the proof of \cite[Lem.\ 3.2]{brad_SchottkyDiffusion}, let
\begin{align*}
\mathcal{H}_\gamma^\alpha u(x)
&=\int_F H_\alpha(\beta x,\gamma y)(u(y)-u(x))\,\mu(y)
\\
&=\mu(F_{\tr})^{-1}
p^{-\alpha\ell(\beta^{-1}\gamma)}
\int_{F_{\tr}}\absolute{\beta x-\gamma y}^{-\alpha}(u(y)-u(x))\,\mu(y)
\end{align*}
for $x\in F_{\tr}$. If $x,y\in F_{\tr}$, then the distance between $x$ and $\gamma y$ is bounded from below and becomes arbitrarily large. Its values are negative rational powers of $p$ in such a way that $y\mapsto\absolute{\beta x-\gamma y}^{-\alpha}$ is a locally constant function on $F_{\tr}$.
Since $u$ is locally constant on $F_{\tr}$ and all its $\Gamma$-translates, it now follows that the integral converges for all $\gamma\in\Gamma$. Since $F_{\tr}$ is compact, it now follows that $\mathcal{H}_\gamma^\alpha u(x)$ is bounded on $F_{\tr}$. Namely,
by the assumption (\ref{assumption_p}), it follows that the number of
$\gamma\in\Gamma\setminus\mathset{\beta}$ 
such that $\mathcal{H}_\gamma^\alpha u(x)$ is fixed, 
is bounded from above by $(2g)^{\ell(\beta)+\ell}$ 
for $\ell=\ell(\gamma)$, 
cf.\ Lemma \cite[Lem.\ 3.1]{brad_SchottkyDiffusion}.
Hence, together with
\[
\ell(\beta^{-1}\gamma)\le\ell(\beta)+\ell(\gamma)
\]
it follows that
$\mathcal{H}_\alpha u(x)$ is a constant times a series consisting of a positive coefficient times non-negative rational powers of 
\[
(2g) p^{-\alpha},
\]
where the rational exponents form an unbounded increasing sequence. By taking the largest integer smaller than or equal to these exponents, one finds a geometric series depending on $x\in F_{\tr}$ as an upper bound for 
$\mathcal{H}_\alpha u(x)$.
This proves the first assertion.

\smallskip
The proof of the second assertion
is similar as the first assertion, just like that of \cite[Lem.\ 3.3]{brad_SchottkyDiffusion}, and
uses
\[
\mathcal{E}_\alpha(u)
=\mu(F_{\tr})^{-2}\sum\limits_{\beta,\gamma\in\Gamma}
p^{-2\alpha\ell(\beta^{-1}\gamma)}
\iint_{F_{\tr}^2}\absolute{\beta x-\gamma y}^{-\alpha}\absolute{u(y)-u(x)}^2\,\mu(y)\mu(x)
\]
together with a similar argument as in the previous case.
\end{proof}

A consequence of Lemma \ref{denselyDefined} is that operator $\mathcal{H}_\alpha$ and Dirichlet forms $\mathcal{E}_\alpha$ are densely defined on $L^2(\Omega_{\tr},\mu)^\Gamma$.
Let
$\mathcal{H}_\alpha^*\colon L^2(\Omega_{\tr},\mu)\to L^2(\Omega_{\tr},\mu)^\Gamma$ be the adjoint of $\mathcal{H}_\alpha\colon L^2(\Omega_{\tr},\mu)^\Gamma\to L^2(\Omega_{\tr},\mu)$. Thus obtain operators
\begin{align*}
\Delta_\alpha&=\mathcal{H}_\alpha^*\circ\mathcal{H}_\alpha\colon L^2(\Omega_{\tr},\mu)^\Gamma\to L^2(\Omega_{\tr},\mu)^\Gamma
\\
\Delta_\alpha^\dagger&=\mathcal{H}_\alpha\circ\mathcal{H}_\alpha^*\colon L^2(\Omega_{\tr},\mu)\to L^2(\Omega_{\tr},\mu)
\end{align*}
for $\alpha>0$ satisfying (\ref{assumption_p}).

\begin{Lemma}
The operators $\mathcal{H}_\alpha,\mathcal{H}_\alpha^*$ are closed, the operators $\Delta_\alpha,\Delta_\alpha^\dagger$ are self-adjoint, and the operators $I+\Delta_\alpha,I+\Delta_\alpha^\dagger$ have bounded inverses for $\alpha>0$ satisfying (\ref{assumption_p}).
\end{Lemma}

\begin{proof}
The proof of \cite[Lem.\ 3.4]{brad_SchottkyDiffusion} carries over in a straightforward manner.
\end{proof}

\subsection{Spectrum of the Schottky invariant operator}

For Schottky invariant diffusion operators defined over the points in a local field of a regular domain for the $p$-adic Schottky group $\Gamma$, the Kozyrev wavelets play an important role in constructing $\Gamma$-invariant
Laplacian eigenfunctions, cf.\ \cite[Thm.\ 4.10]{brad_SchottkyDiffusion}. However, this is not going to be feasible on the transcendent $p$-adic points, because their local profinite structure is not regular, meaning that these are more general ultrametric kinds of ``discs'' than the regular structures one obtains from $p$-adic discs. For this reason, the ultrametric wavelets defined in \cite{XK2005} will play a substitute role in the present case.
\newline

In the following, it will be assumed that the local equity measure $\nu_F$ on $F_{\tr}$ satisfies the following property:
\begin{align}\label{diameterProperty}
\nu_F(A)=\diam(A)
\end{align}
for any disc $A\subset F_{\tr}\subset\Omega_{\tr}\subset\mathds{C}_p$, i.e.\ the diameter is determined via the $p$-adic absolute value $\absolute{\cdot}$. For the measure $\mu$ on $\Omega_{\tr}$, this means that
representing the differential $1$-form $\omega$ locally as
\[
\omega|_{F}=f\,dx
\]
for some analytic function $f$ on $F$, 
allows to conclude:
\begin{Lemma}
It holds true that
\[
\mu(A)=C_A\cdot\nu_F(A)
\]
with $C_A>0$ for any disc $A\subset F_{\tr}$.
\end{Lemma}

\begin{proof}
It holds true that  
\[
\mu(A)=\int_{A}\absolute{\omega}_{\tr}
=\int_A\absolute{f(x)}\,\nu_F(x)
\]
with 
\[
A=\mathset{x\in\Omega_{\tr}\mid\absolute{x-a}\le \absolute{r}}
\]
with $r\in\mathds{C}_p$.
Since $f$ is a non-vanishing analytic function 
 on the disc
 \[
 B =
 \mathset{x\in\mathds{C}_p\mid\absolute{x-a}\le\absolute{r}}
 \]
 satisfying $A=B\cap\Omega_{\tr}$, it follows that $f(B)\subset\mathds{C}_p$ is a closed disc,  cf.\ \cite[Prop.\ 2.3]{Benedetto2001}. Since this disc does not contain $0\in\mathds{C}_p$, it follows that it is  contained in a circle centred in $0$. Hence, $\absolute{f(x)}$ is constant on $A$. This proves the assertion.
\end{proof}

Because 
\[
\mu(\gamma A)=\int_{\gamma A}\absolute{\omega}_{\tr}=\int_A\frac{\absolute{f(x)}}{\absolute{\gamma'(x)}}\,\nu_F(\gamma x)
=\int_A\absolute{\omega}_{\tr}=\mu(A)
\]
for $\gamma\in\Gamma$, it
is $\Gamma$-invariant, and hence the constant $C_A>0$ is invariant under $\Gamma$.
Notice that by the analytic property of $\gamma\in\Gamma$, the image $\gamma A\subset\Omega_{\tr}$ is also a disc.

\begin{Lemma}
It holds true that
\[
\absolute{\beta x-\gamma y}=\absolute{x-\beta^{-1}\gamma y}
\]
for $x,y\in F_{\tr}$ and $\beta,\gamma\in\Gamma$.
\end{Lemma}

\begin{proof}
The proof of \cite[Lem.\ 4.8]{brad_SchottkyDiffusion} carries over to this case.
\end{proof}

In the $p$-adic case, the Kozyrev wavelets \cite{Kozyrev2004} turn out to be eigenfunctions of certain integral operators. In the present case, this role is taken by the ultrametric wavelets from \cite{XK2005} which can be defined with the help of an equity measure by representing the $m(v)$ children of any given node $v$ of a rooted tree by a cyclic group $C_{m(v)}$. An ultrametric wavelet is  supported on the boundary of such a tree and is constant on the boundary of the branch beginning in each child node of the root vertex a unit root, taking its value from evaluating a unitary character of $C_{m(v)}$.
This approach has been used in \cite[Def.\ 3.11]{BL-topoIndex_p} in order to be able to approximate graph Laplacians with $p$-adic Laplacians having a a  kernel function defined by an ultrametric.
\newline

The following Lemma generalises the observation of \cite[Lem.\ 3.12]{BL-topoIndex_p}:

\begin{Lemma}\label{meanZeroWavelet}
Let $\psi$ be an ultrameric wavelet supported in a disc $A\subset F_{\tr}$. Then
\[
\int_A\psi(x)\mu(x)=0
\]
holds true.
\end{Lemma}

\begin{proof}
This follows immediately from the local equity property of the measure $\mu$ on $\Omega_{\tr}$.
\end{proof}

The following result  corresponds with \cite[Lem.\ 4.9]{brad_SchottkyDiffusion}:

\begin{Lemma}
It holds true that
\begin{align*}
\int_{F_{\tr}}&
\absolute{x-y}^{-\alpha}(\psi_A(y)-\psi_A(x))\mu(y)
\\
&=-\left(
\int_{F_{\tr}\setminus A}\absolute{x-y}^{-\alpha}\mu(y)+\mu(A)^{1-\alpha}
\right)\psi_A(x)
\end{align*}
for $\alpha>0$ satisfying (\ref{assumption_p}) with $\psi_A$ an ultrametric wavelet supported in a disc $A\subset F_{\tr}$.
\end{Lemma}

\begin{proof}
This follows from
\cite[Thm.\ 10]{XK2005}, since condition \cite[eq.\ (21)]{XK2005} is satisfied for the kernel function.
\end{proof}

Just like in \cite{brad_SchottkyDiffusion}, an ultrametric wavelet $\psi_A$ supported in $A\subset F_{\tr}$ can be extended to a $\Gamma$-invariant function $\psi_A^\Gamma$
on $\Omega_{\tr}$ via
\[
\psi_A^\Gamma(\gamma x)=\psi_A(x)
\]
for $\gamma\in\Gamma$, called \emph{$\Gamma$-invariant ultrametric wavelet}. Of interest are those wavelets $\psi_A^\Gamma$ supported on the $\Gamma$-translates of discs $A\subset F_{\tr}$ such that
\begin{align}\label{discTr}
A=B\cap\Omega_{\tr}
\end{align}
with $B$ a disc of equal radius as $A$ inside $F_{\tr}\setminus V(\omega)$. The $p$-adic disc $B$ of equal radius as $A$ satisfying (\ref{discTr})
is called a \emph{$\mathds{C}_p$-extension} of $A$. The orbit $\Gamma A$ is called an \emph{invariant disc} in $\Omega_{\tr}$, and $\Gamma B$ an \emph{invariant $\mathds{C}_p$-extension} of $A$.
Define
\begin{align*}
\Omega_{\tr}^\omega&=
\mathset{x\in\Omega_{\tr}\mid
\text{\begin{minipage}{6cm}$\exists$ disc $A\ni x$ having an invariant $\mathds{C}_p$-extension $\Gamma B\subset\Omega\setminus V(\omega)$
\end{minipage}}
}
\\
F_{\tr}^\omega&=F\cap\Omega_{\tr}^\omega
\end{align*}
where $\omega$ is the given regular $\Gamma$-invariant differential $1$-form on $\Omega$.

\begin{thm}\label{MumfordEV}
The space $L^2(\Omega_{\tr}^\omega,\mu)^\Gamma$ of $\Gamma$-invariant $L^2$-functions on $\Omega_{\tr}^\omega$ has an orthonormal basis consisting of the $\Gamma$-invariant ultrametric wavelets $\psi_A^\Gamma$ supported in $\Omega_{\tr}^\omega$, and these are eigenfunctions of $\Delta_\alpha^{\frac12}$ for $\alpha>0$. The eigenvalue corresponding to $\psi_A^\Gamma$ is
\begin{align*}
\lambda_A^{(\alpha)}&=-\mu(F_{\tr}^\omega)^{-1}
\sum\limits_{\gamma\in\Gamma}
p^{-\alpha\ell(\gamma)}
\left(
\int_{F_{\tr}^\omega\setminus A}
\absolute{x-\gamma y}^{-\alpha}\mu(y)+\mu(A)^{1-\alpha}
\right)
\end{align*}
for $A\subset F_{\tr}^\omega$, and $F\subset\Omega$ a good fundamental domain for the action of $\Gamma$. Here, $x\in A$, and $\lambda_A$ does not depend on $x\in A$. The multiplicity of $\lambda_A$ is finite. Both, $\lambda_A$ and its multiplicity, are invariant under replacing $F$ with $\gamma F$ for $\gamma\in\Gamma$. The restriction of $\Delta_\alpha^{\frac12}$ to $L^2(\Omega_{\tr}^\omega)^\Gamma$ coincides with $-\mathcal{H}_\alpha$ for $\alpha>0$ satisfying (\ref{assumption_p}).
\end{thm}

This theorem corresponds with \cite[Thm.\ 4.10]{brad_SchottkyDiffusion}.

\begin{proof}
The proof of \cite[Thm.\ 4.10]{brad_SchottkyDiffusion} now carries over, because it can use the analogous Lemmas stated prior to this theorem in a similar way. 
\end{proof}

\subsection{Invariant diffusion operators for Shimura curves}

A $p$-adically uniformised Shimura curve can be viewed as a locally ultrametric space with disjoint pieces invariant under a Schottky group. This gives rise to new Laplacian operators.
\newline

Now, that Corollary \ref{countable} showed that $\Omega_{\tr}$ is a countable union of Galois orbits, the theory of Schottky-invariant Z\'u\~niga-type operators from \cite{brad_SchottkyDiffusion}
can be adapted to the situation of a $p$-adically uniformised Shimura curve.
\newline

Assume that the Shimura curve $S_N^{\rig}$ is defined over $\mathds{Q}_p$ and has its semi-stable reduction  also defined over $\mathds{Q}_p$. The latter is a necessary technicality
from algebraic geometry, and can be achieved by a sufficiently large finite
field extension of $\mathds{Q}_p$, if necessary, cf.\ e.g.\ \cite[Thm.\ 5.5.3]{FP2004}. 
Given the transcendent $p$-adic points of a Shimura curve as in (\ref{GivenShimura}), let $\omega_1,\dots,\omega_N\in\Omega^1(\Omega_p)$ with $\omega_i$ being a $\Gamma_i'$-invariant non-vanishing regular differential $1$-form on $\Omega_p$. 
This is possible, since the Shimura curve is defined over $\mathds{Q}_p$.
Define
the following spaces:
\begin{align*}
\Omega_i&=\Omega\times\mathset{i},&\Omega_{i,\tr},=\Omega_{\tr}\times\mathset{i}
\\
\Omega'&=\bigsqcup\limits_{i=1}^N\Omega_i,&\Omega'_{\tr}=\bigsqcup\limits_{i=1}^N\Omega_{i,\tr}\,,
\end{align*}
and fix a $\Gamma$-invariant local equity measure $\absolute{\omega_i}_{\tr}$ on $\Omega_i'$ as in Proposition \ref{GammaLocalEquity} for $i=1,\dots, N$. Then define
\[
\mu_{\tr}=\absolute{\omega_1}_{\tr}\sqcup\dots\sqcup\absolute{\omega_N}_{\tr}
\]
as their co-product measure on $\Omega'_{\tr}$. Further, define
\begin{align*}
\Omega_{i,\tr}^{\omega_i}&=\Omega_{\tr}^{\omega_i}\times\mathset{i},
&\Omega_{\tr}^{\omega}=\bigsqcup\limits_{i=1}^N\Omega_{i,\tr}^{\omega_i}
\end{align*}
as well as
\[
\underline{\Gamma}=\Gamma_1'\times\cdots\times\Gamma_N'
\]
and
\[
F'=\bigsqcup\limits_{i=1}^NF_i,\quad F^\omega=F'\cap\Omega^\omega_{\tr}
\]
where $F_i$ is the transcendent part of a good fundamental domain in $\Omega_i$. The subpspace of a function space 
\[
\mathcal{F}(\Omega^\omega_{\tr})=\mathcal{F}(\Omega_{1,\tr}^{\omega_1})\oplus\cdots\oplus\mathcal{F}(\Omega_{N,\tr}^{\omega_N})
\]
invariant under $\underline{\Gamma}$ is denoted as
and understood as
\[
\mathcal{F}(\Omega^\omega_{\tr})^{\underline{\Gamma}}=\bigoplus\limits_{i=1}^N
\mathcal{F}(\Omega_{i,\tr}^{\omega_i})^{\Gamma_i'}
\]
where each summand is the subspace of the corresponding summand above invariant under the corresponding Schottky group. An exception is the notation $\mathcal{D}(\Omega_{i,\tr}^{\omega_i})^{\Gamma_i'}$ and $\mathcal{D}(\Omega_{\tr}^\omega)^{\underline{\Gamma}}$ which stands just for locally constant and invariant functions. These are extensions of compactly supported functions on a good fundamental domain.
\newline

Now, define an integral operator for $\alpha_1,\dots,\alpha_N>0$ each satisfying the corresponding condition (\ref{assumption_p}) as follows:

\begin{align}\label{ShimuraLaplacian}
\mathcal{L}\colon L^2(\Omega_{\tr}^\omega,\mu_{\tr})^{\underline{\Gamma}}\to L^2(\Omega_{\tr}^\omega,\mu_{\tr})^{\underline{\Gamma}},\;u\mapsto\mathcal{Z}u
-\sum\limits_{i=1}^N\Delta_{i,\alpha_i}^{\frac12}
\end{align}
with
$\Delta_{i,\alpha_i}^{\frac12}$
equal to the operator $\Delta_{\alpha_i}^{\frac12}$ on $L^2(\Omega_{i,\tr}^{\omega_i},\mu_i)^{\Gamma_i'}$, 
\[
\mathcal{Z}u(x)=\int_{\Omega'_{\tr}} Z(x,y)(u(y)-u(x))\,d\mu_{\tr}(y)
\]
and kernel function
\[
Z(x,y)=\sum\limits_{i=1\atop i\neq j}^N
\sum\limits_{\gamma_i\in\Gamma_i'}\sum\limits_{\gamma_j\in\Gamma_j'}
w(\gamma_i x,\gamma_j y)\,
\Omega(x\in F_{i})\Omega(y\in F_{j})
\]
for $x,y\in F'$,
$F_i\subset\Omega_{i,\tr}$ a good fundamental domain for $\Gamma_i'$,
and 
\[
w(\gamma_i x,\gamma_j y)=
p^{-\alpha_Z(\ell(\gamma_i)+\ell(\gamma_j))}w_{ij}
\]
with $\alpha_Z>>0$,
$\gamma_k\in\Gamma_k'$, $w_{ij}\ge0$ and $w_{ij}=w_{ji}$ for $i\neq j$ and $i,j=1,\dots,N$.

\begin{Lemma}\label{ZunigaEV}
Let $\psi_A^{\Gamma_{i_0}'}$ be a $\Gamma_{i_0}$-invariant ultrametric wavelet supported in the orbit of $A\subset\Omega_{i_0,\tr}^{\omega_{i_0}}$. 
Assuming that $\alpha_Z>>0$, this function is an eigenfunction of $\mathcal{Z}$ with eigenvalue
\[
\lambda_{i_0}=-\sum\limits_{j=1\atop j\neq i_0}^N
\sum\limits_{\gamma_i\in\Gamma_{i_0}'}
\sum\limits_{\gamma_j\in\Gamma_j'}
p^{-\ell(\gamma_i)-\ell(\gamma_j)}
w_{i_0,j}
\,\mu(F_j)
\]
which only depends on $i_0\in\mathset{1,\dots,N}$.
\end{Lemma}

\begin{proof}
Since the number of group elements in a Schottky group having bounded length is bounded by a function of the genus, as in \cite[Lem.\ 3.1]{brad_SchottkyDiffusion}, there exists $\alpha_Z>>0$ such that
\begin{align*}
\mathcal{Z}\psi_A(x)&=
\sum\limits_{i,j=1\atop i\neq j}^N
\sum\limits_{\gamma_i\in\Gamma_i'}
\sum\limits_{\gamma_j\in\Gamma_j'}
w(\gamma_i x,\gamma_j y)
\int_{F_j}(\psi_A(y)-\psi_A(x))\mu_{\tr}(y)\,\Omega(x\in F_i)
\\
&=\sum\limits_{j=0\atop j\neq i_0}^N
\sum\limits_{\gamma_i\in\Gamma_{i_0}'}
\sum\limits_{\gamma_j\in\Gamma_j'}
p^{-\alpha_Z(\ell(\gamma_i)+\ell(\gamma_j))}
w_{i_0,j}
\left(\int_{F_j}
(\psi_A(y)\mu(y)-\mu(F_j)\psi_A(x)\mu_{\tr}(y)\right)\cdot
\\
&\cdot\Omega(x\in F_{i_0})
\\
&\stackrel{\text{Lem.\ \ref{meanZeroWavelet}}}{=}
-\sum\limits_{j\neq i_0}
\sum\limits_{\gamma_i\in\Gamma_{i_0}'}
\sum\limits_{\gamma_j\in\Gamma_j'}
p^{-\alpha_Z(\ell(\gamma_i)+\ell(\gamma_j))}
w_{i_0,j}\mu(F_j)\,\psi_A(x)
\end{align*}
where the eigenvalue is as asserted, and does not depend on $A\subset\Omega_{i,\tr}$.
\end{proof}

\begin{thm}\label{ShimuraEigenbasis}
The space $L^2(\Omega_{\tr}^\omega,\mu_{\tr})^{\underline{\Gamma}}$ has an orthonormal basis consisting of ultrametric wavelets which are invariant under one of the Schottky groups $\Gamma_i'$ with $i=1,\dots,N$. These are eigenfunctions of $\mathcal{L}$ with eigenvalue
\[
\lambda^{(\alpha_i)}_A+\lambda_{i}
\]
where the wavelet is supported in the orbit of $A\subset\Omega_{i,\tr}^{\omega_i}$, and $\lambda_A$ is the corresponding eigenvalue w.r.t.\ operator $\Delta_{i,\alpha_i}$, and $\lambda_i$ the one w.r.t.\ to operator $\mathcal{Z}$. The operator $\mathcal{L}$ is self-adjoint, and all its eigenvalues are negative and have finite multiplicity.
\end{thm}

\begin{proof}
The orthonormal eigenbasis property and the
eigenvalues are an immediate consequence of 
Theorem \ref{MumfordEV}
and Lemma \ref{ZunigaEV}.
This implies the self-adjointness of $\mathcal{L}$ and the negativity of the eigenvalues. The finite multiplicity of eigenvalues then follows from Theorem \ref{MumfordEV}.
\end{proof}

\section{Heat Equations for Shimura Curves}

The goal here is to study the heat equation
\begin{align}\label{ShimuraHeatEquation}
\frac{\partial }{\partial t}u(x,t)-\mathcal{L}u(x,t)=0
\end{align}
with $\mathcal{L}$ the Laplacian integral operator (\ref{ShimuraLaplacian}). First, the corresponding Cauchy problem and heat kernel. This is completed by a study of associated boundary value problems.

\subsection{Feller semigroup, Cauchy problem and heat kernel}

\begin{Lemma}
The linear operator $\mathcal{L}$ generates a Feller semigroup $\exp(t\mathcal{L})$ with $t\ge0$ on $C(\Omega_{\tr}',\mathds{R})^{\underline{\Gamma}}$.
\end{Lemma}

\begin{proof}
The criteria for applying the Hille-Yosida-Ray Theorem are verified:

\smallskip
1. The domain of $\mathcal{L}$ is dense in $C(\Omega_{\tr}',\mathds{R})^{\underline{\Gamma}}$. This follows from Lemma \ref{denselyDefined} and $Z(x,\cdot)\in \mathcal{D}(\Omega_{\tr}')^{\underline{\Gamma}}$ for every $x\in\Omega_{\tr}'$.

\smallskip
2. The operator $\mathcal{L}$ satisfies the positive maximum principle. For this let $h\in\mathcal{D}(\Omega_{\tr}',\mathds{R})^{\underline{\Gamma}}$ and $x_0\in\Omega_{\tr}'$ such that $h(x_0)=\sup\limits_{x\in\Omega_{\tr}'}h(x)$. This exists, because $h$ is periodic and $F$ is compact. 
Then
\begin{align*}
\mathcal{L}h(x_0)&=\int_{\Omega_{\tr}'}\left(Z(x_0,y)+\sum\limits_{i=1}^N H_i(x_0,y)\right)(h(x_0)-h(x_0))\,\mu_{\tr}(y)\le0,
\end{align*}
where $H_i(x,y)$ is the kernel function of $-\Delta_{i,\alpha_i}^{\frac12}$. This implies the positive maximum principle.

\smallskip
3. The range $\Ran(\eta I-\mathcal{L})$ is dense in $C(\Omega_{\tr}',\mathds{R})^{\underline{\Gamma}}$. The proof is analogous to the one contained in the proof of \cite[Lem.\ 5.1]{brad_SchottkyDiffusion}. For the convenience of the reader, the adaptations to the situation here are given now. Let $h\in C(\Omega_{\tr}',\mathds{R})^{\underline{\Gamma}}$, and $\eta>0$. The task is to find a solution of the equation
\begin{align}\label{resolventEquation}
(\eta I-\mathcal{L})u=h
\end{align}
for some $\eta>0$ and $h$ in some dense subspace of $C(\Omega_{\tr}',\mathds{R})^{\underline{\Gamma}}$. The equation can formally be rewritten as
\begin{align}\label{formalRewrite}
u(z)-\frac{\int L(x,y)u(y)\mu_{\tr}(y)}{\eta+\deg(z)}=\frac{h(z)}{\eta+\deg(z)}\,,
\end{align}
where $L(x,y)$ is the kernel function of $\mathcal{L}$, and
\[
\deg(z)=\int_{\Omega_{\tr}'}L(z,y)\mu_{\tr}(y)
\]
which, of course, does not converge. For this reason, take the operator
\[
T_ku(z)=\frac{\int_{\Omega_{\tr,z}^{(k)}} L(z,y)u(y)\mu_{\tr}(y)}{\eta+\deg_k(z)}
\]
with
\[
\Omega^{(k)}_{\tr,z}=\Omega_{1,\tr}\sqcup\cdots
\sqcup\Omega_{i,z}^{(k)}\sqcup\cdots\sqcup\Omega_{N,\tr}\,,
\]
where 
\[
\Omega_{i,z}^{(k)}=\bigsqcup\limits_{\gamma\in\Gamma_i'}\gamma F_{i,z}^{(k)}
\]
and
\[
F_{i,z}^{(k)}=F_{i}\setminus B_k(z)
\]
where w.l.o.g.\ it is assumed that $z\in F_{i}$, and such that
\[
B_0(z)\supset B_1(z)\supset B_2(z)\supset\dots
\]
is the maximal chain of discs contained in $F_i$ and containing $z$. 
Let 
\[
\deg_k(z)=\int_{\Omega_{\tr,z}^{(k)}}L(z,y)\mu_{\tr}(y)
\]
which does converge, and
\[
\absolute{T_ku(z)}\le\frac{\deg_k(z)}{\eta+\deg_k(z)}\norm{u}_\infty\,,
\]
where $\norm{u}_\infty$ is finite, because $u$ is invariant, and $F$ is compact. Hence,
\[
\norm{T_k}\le\frac{1}{\eta/\deg_k(z)+1}<1
\]
for any $\eta>0$ and $k\in\mathds{N}$. It follows that $I+T_k$
has a bounded inverse as an operator on $C(\Omega_{\tr}',\mathds{R})^{\underline{\Gamma}}$. This means that its range is dense for $k\in\mathds{N}$.

\smallskip
Now, let $h\in\mathcal{D}(\Omega_{\tr}',\mathds{R})^{\underline{\Gamma}}$, and let $u_k,u_\ell\in C(\Omega_{\tr}',\mathds{R})^{\underline{\Gamma}}$ be solutions of 
\[
(I+T_k)u_k=\frac{h}{\eta+\deg_k},\quad(I+T_\ell)u_\ell=\frac{h}{\eta+\deg_\ell}
\]
for $k,\ell\in\mathds{N}$. Then
\begin{align}\label{CauchySeq}
u_k-u_\ell=\frac{(I+T_\ell)(\eta+\deg_\ell)-(I+T_k)(\eta+\deg_k)}{(I+T_k)(I+T_\ell)(\eta+\deg_k)(\eta+\deg_\ell)}\,h
\end{align}
shows that $u_k$ is a Cauchy sequence w.r.t.\ $\norm{\cdot}_\infty$. The reason is that, firstly,
\[
\norm{T_k}=\sup\limits_{z\in F}\frac{\deg_k(z)}{\eta+\deg_k(z)}
\]
holds true, and this is a strictly increasing sequence converging to $1$. This implies that $T_k$ converges to a bounded operator $T$ on $C(\Omega_{\tr}',\mathds{R})^{\underline{\Gamma}}$. Secondly, the numerator of the right hand side of (\ref{CauchySeq}) is
\[
\eta(T_\ell-T_k)+(\deg_\ell-\deg_k)+(T_\ell\deg_\ell-T_k\deg_k)
\]
whose first and third summands become arbitrarily small in norm as $\ell\ge k\to\infty$. The third term is 
\[
T_\ell\deg_\ell-T_k\deg_k=(T_\ell\deg_\ell-T_k\deg_\ell)+(T_k\deg_\ell-T_k\deg_k)
\]
for which both summands also become arbitrarily small in norm as $\ell\ge k\to\infty$.
Hence, $u_k$ converges to some $u\in C(\Omega_{\tr}',\mathds{R})^{\underline{\Gamma}}$ which can be seen to be a solution of (\ref{resolventEquation}) as follows:
\[
\lim\limits_{k\to\infty}(\eta+\deg_k)T_k=(\eta+\deg) T
\]
with convergence in norm. The limit operator coincides with the unbounded operator
\[
u\mapsto Au=\int_{\Omega_{\tr}'} L(\cdot,y)u(y)\mu_{\tr}(y)
\]
which shows that the operator 
\[
T=\frac{A}{\eta+\deg}
\]
appearing in (\ref{formalRewrite}) is bounded. Now, $u_k$ is a solution of
\[
(\eta-L_k)u_k=h
\]
with
\[
L_k=(\eta+\deg_k)T_k-\deg_k
\]
which converges to $\mathcal{L}$ for $k\to\infty$. As $u_k\to u$, it follows that 
\[
(\eta I+\ell)u=(\eta I-L_k)u+(L_k-\mathcal{L})u,
\]
where
\[
(\eta I-L_k)u=(\eta I-L_k)u_k+L_k(u_k-u)=h+L_k(u_k-u)\to h
\]
and
\[
(L_k-\mathcal{L})u\to0
\]
for $k\to\infty$. Hence, $u$ is a solution of (\ref{resolventEquation}). This shows that $\Ran(\eta I-\mathcal{L})$ contains the invariant real-valued test functions, and is thus dense in $C(\Omega_{\tr}',\mathds{R})^{\underline{\Gamma}}$.

\smallskip
The assertion now follows with Hille-Yosida-Ray.
\end{proof}

\begin{thm}\label{heatKernelDistribution}
There exists a probability measure $p_t(x,\cdot)$ with $t\ge0$, $x\in\Omega_{\tr}^\omega$ on the Borel $\sigma$-algebra of $\Omega_{\tr}^\omega$ such that the Cauchy problem for the heat equation (\ref{ShimuraHeatEquation}) with initial condition $u(x,0)=u_0(x)\in C(\Omega_{\tr}^\omega,\mathds{R})^{\underline{\Gamma}}$ has a unique solution in $C^1((0,\infty),\Omega_{\tr}^\omega)^{\underline{\Gamma}}$ of the form 
\[
u(x,t)=\int_{\Omega_{\tr}^\omega} u_0(y)p_t(x,\mu_{\tr}(y))
\]
Additionally, $p_t(x,\cdot)$ is the transition function of a strong Markov process on $\Omega_{\tr}^\omega/\underline{\Gamma}$, where the action of $\underline{\Gamma}$ is by $\Gamma_i'$ on $\Omega_{i,\tr}^{\omega_i}$ for each $i=1,\dots,N$, and whose paths are c\`adl\`g.  
\end{thm}

\begin{proof}
The proof is the same as for \cite[Thm.\ 5.2]{brad_SchottkyDiffusion}.
\end{proof}

To say that a diffusion process has a heat kernel function means that there exists a distribution function $p(t,x,y)$ for the Markov distribution $p_t(x,\mu_{\tr}(y))$ from Theorem \ref{heatKernelDistribution} in the sense that
\[
p_t(x,\mu_{\tr}(A))=\int_Ap(t,x,y)\,\mu_{tr}(y)
\]
for $t>0$, $x\in\Omega_{\tr}^\omega$, and $A\subset\Omega_{\tr}^\omega$ a measurable set.
It may happen that $p_t(x,x)$ is not defined for all $t\ge0$. This is the case in the present setting.

\begin{cor}\label{heatKernel}
The heat equation 
(\ref{ShimuraHeatEquation})
has a well-defined heat kernel function $p(t,x,y)$ for $t\ge0$ and $x\neq y$. In the case $x=y$ it is only defined for $t>>0$.
\end{cor}

\begin{proof}
Due to Theorems \ref{ShimuraEigenbasis} and \ref{heatKernelDistribution}, it can be said that
if it exists, then the heat kernel has the following decomposition:
\begin{align}\label{ansatz}
p(t,x,y)=\sum\limits_{\psi}e^{\lambda_\psi t}\psi(x)\overline{\psi(y)}
\end{align}
where $\psi$ runs over an orthonormal basis of $L^2(\Omega_{\tr}^\omega,\mu_{\tr})^{\underline{\Gamma}}$,
and the question is whether this expression always converges. First, assume that $x\neq y$. Then, since the eigenfunctions are wavelets supported in one of the constituent $\Omega_{i,\tr}^{\omega_i}$, it follows that if $x$ and $y$ do not belong to the same $\Omega_{i,\tr}^{\omega_i}$, then the sum (\ref{ansatz}) vanishes. So, assume now that $x,y\in\Omega_{i,\tr}^{\omega_i}$. By invariance, one may assume that $x,y$ belong to $F_{i,\tr}^{\omega_i}$. This means that the sum (\ref{ansatz}) contains only those wavelets whose support intersection with the good fundamental domain $F_i$ contains the ultrametric disc belonging to the join of $x$ and $y$. These are only finitely many because of the compactness of $F_i$. Hence, the sum converges for all $t\ge0$ and $x\neq y$.

\smallskip
Now, assume that $x=y$. Then (\ref{ansatz}) becomes
\[
p(t,x,x)=\sum\limits_{\psi}e^{\lambda_\psi t}\mu_{\tr}(\supp(\psi))^{-1}
\]
and converges if and only if  $t>>0$ because of the finite dimensionality of the eigenspaces of $\mathcal{L}$, and because the spectrum of $-\mathcal{L}$ has a smallest positive element.
\end{proof}

\subsection{Boundary Value Problems}

In order to say something about boundary value problems involving the heat equation (\ref{ShimuraHeatEquation}), it is helpful to use separation of variables:
\[
u(x,t)=V(x)W(t)
\]
and obtain
\[
\frac{\dot{W}(t)}{W(t)}
=\frac{\mathcal{L}U(x)}{U(x)}
\]
which leads to the system
\begin{align*}
\mathcal{L}V(x)&=\lambda V(x)
\\
\dot{W}(t)&=\lambda W(t)
\end{align*}
giving rise to the solution
\begin{align}\label{SV_solution}
u(x,t)=e^{\lambda_\psi t}\psi(x)
\end{align}
for a $\Gamma_i'$-invariant ultrametric wavelet supported in some $\Omega_{i,\tr}^\omega$ for some $i=1,\dots,N$. Since here the focus is on invariant Dirichlet and von Neumann boundary conditions, it suffices to generalise the corresponding $p$-adic definitions found in \cite[\S 2.5]{NonAutonomous}.
\newline

The \emph{vertex boundary} of $S$ is
\[
\delta S=\mathset{x\in F_{\tr}^\omega\setminus S\mid \exists\, y\in S\colon L(x,y)\neq0}
\]
and the \emph{edge boundary} of $S$ is
\[
\partial S=\mathset{
(x,y)\in \left(F_{\tr}^{\omega}\right)^2\mid
x\in S\;\wedge\;y\notin S}
\]
for $S\subset F_{\tr}^\omega$ compact open. 
Its invariant versions are
\[
\underline{\Gamma}\delta S
=\bigsqcup\limits_{i=1}^N\bigsqcup\limits_{\gamma\in\Gamma_i'}\gamma(\delta S\cap F_{i,\tr}^{\omega_i})
\]
and
\[
\underline{\Gamma}\partial S:=\mathset{(x,y)\in \left(\Omega_{\tr}^\omega\right)^2\mid
x\in\underline{\Gamma} S\;\wedge\; y\notin \underline{\Gamma}S}
\]
with an obvious extension of the interpretation of an orbit of $S$ under $\underline{\Gamma}$.
For an invariant function $f\colon \Omega_{\tr}^\omega\to\mathds{C}$ to satisfy the \emph{invariant Dirichlet boundary condition} means
\[
\forall x\in \underline{\Gamma}\delta S\colon f(x)=0\,,
\]
and for it to satisfy the \emph{invariant von Neumann boundary condition} means
\[
\forall x\in\underline{\Gamma}\delta S\colon
\int_{\mathset{(x,y)\in\underline{\Gamma}\partial S}} L(x,y)(f(y)-f(x))\mu_{\tr}(y)=0\,,
\]
where $L(x,y)$ is the kernel function of the integral operator $\mathcal{L}$.

\begin{thm}\label{bvp_Shimura}
The boundary value problem
\begin{align*}
\frac{\partial}{\partial t}
u(x,t)&-\mathcal{L}u(x,t)=0
\\
u(x,0)&=u_0(x)\in\mathcal{D}(\Omega_{\tr}^\omega)^{\underline{\Gamma}}
\end{align*}
with $u(x,t)$ satisfying either the invariant Dirichlet or the invariant von Neumann boundary condition for $S\subset F_{\tr}^\omega$ compact open and for all $t\ge0$ has a solution if and only if $u_0(x)$ is supported in $\underline{\Gamma}S$.  
\end{thm}

\begin{proof}
\emph{Invariant Dirichlet boundary condition.}
If $u_0$ is supported in $\underline{\Gamma}S$, then from Corollary \ref{heatKernel} it follows that the solution $u(x,t)$ of the Cauchy problem remains supported in $\underline{\Gamma}S$ for all $t\ge0$. Hence, the invariant Dirichlet boundary is satisfied. Conversely, from (\ref{SV_solution}), it follows that imposing the invariant Dirichlet boundary condition restricts the eigenfunctions $\psi$ to be supported in $\underline{\Gamma} S$ at all times $t\ge0$. Hence, the initial condition must be supported in $\underline{\Gamma}S$.

\smallskip
\emph{Invariant von Neumann boundary condition.} This condition means
\begin{align*}
0&=\int_{\underline{\Gamma}S} L(x,y)(f(y)-f(x))\mu_{\tr}(y)
\\
&=\int_{\underline{\Gamma}S}L(x,y)f(y)\mu_{\tr}(y)-\int_SL(x,y)\mu_{\tr}(y)\,f(x)
\end{align*}
for $x\notin S$. These two  integrals converge, because $y$ never gets too close to $x$, and the defining properties of $L(x,y)$. If the first integral vanishes, then due to the vanishing of the second integral, it follows that $f$ is supported in $\underline{\Gamma}S$. In this case, (\ref{SV_solution}) says that the solution $u(x,t)$ is confined to $\underline{\Gamma}S$ at all times $t\ge0$. If, on the other hand, the second interval does not vanish for $x\notin S$, then the expansion of $f$ w.r.t.\ the ultrametric wavelet basis also contains wavelets supported outside of $S$. But for such a wavelet $\psi$, the first integral vanishes, because  of Lemma \ref{meanZeroWavelet}, since $L(x,y)$ for $y\in S$ only depends on the distance to the support of $\psi$ intersected with $F_{\tr}^\omega$. It follows that 
the possibility of a solution having support outside $\underline{\Gamma}S$ is ruled out.
This proves the assertion also in this case.
\end{proof}

\section*{Acknowledgements}

Paulina Halwas,
Frank Herrlich, Stefan K\"uhnlein, 
\'Angel Mor\'an Ledezma,  Leon Nitsche, David Weisbart and Wilson Z\'u\~niga-Galindo are warmly thanked for fruitful discussions over an extended course of time.
This work is partially supported by the Deutsche Forschungsgemeinschaft under project number 469999674.

\bibliographystyle{plain}
\bibliography{biblio}

\end{document}